\documentclass[a4paper,12pt]{article}

\usepackage{amsthm}
\usepackage{amssymb}
\usepackage{enumerate}
\usepackage{hyperref}
\usepackage{graphicx}
\usepackage{stackrel}
\usepackage{mathtools}
\usepackage[a4paper,margin=2.4cm]{geometry}

\theoremstyle{plain}

\newtheorem{theorem}{Theorem}[section]
\newtheorem{proposition}[theorem]{Proposition}
\newtheorem{lemma}[theorem]{Lemma}  
\newtheorem{corollary}[theorem]{Corollary}

\theoremstyle{remark}

\newtheorem{observation}[theorem]{Observation}
\newtheorem{remark}[theorem]{Remark}

\theoremstyle{definition}

\newtheorem{definition}[theorem]{Definition}

\newcommand{\norm}[1]{\| #1 \|}

\title{On extremal eigenvalues of the graph Laplacian}
\author{Andrea Serio\footnote{\noindent Department of Mathematics, Stockholm University, 10691 Stockholm, Sweden. \newline Email: \texttt{andrea.serio@math.su.se}.}}
\date{4th August 2020}

\begin{document}

\maketitle

\begin{abstract}
    Upper and lower estimates of eigenvalues of the Laplacian on a metric graph have been established in 2017 by G.~Berkolaiko, J.B.~Kennedy, P.~Kurasov and D.~Mugnolo.
    Both these estimates can be achieved at the same time only by highly degenerate eigenvalues which we call \emph{maximally degenerate}.
    By comparison with the maximal eigenvalue multiplicity proved by I.~Kac and V.~Pivovarchik in 2011 we characterize the family of graphs exhibiting maximally degenerate eigenvalues which we call \emph{lasso trees}, namely graphs constructed from trees by attaching lasso graphs to some of the vertices.
\end{abstract}

\section*{Overview}

We are interested in the study of two bounds of the eigenvalues of the graph Laplacian proven in \cite{BKKM17} which depend on a few simple geometrical and topological properties of the graph, namely the total length $\mathcal{L}$, the number of Dirichlet and Neumann pendant vertices $\mathcal{D}$ and $\mathcal{N}$, and the first Betti number $\beta$.
\begin{align}
    \label{lwb}
    \lambda_n \geq {{m}}_n
    &= 
    \begin{cases}
        \frac{\pi^2}{\mathcal{L}^2} \frac{n^2}{4}                                       &\text{if } n < \mathcal{N} + \beta, \\
        \frac{\pi^2}{\mathcal{L}^2} \left( n - \frac{\mathcal{N} + \beta}{2}\right)^2   &\text{if } n \geq \mathcal{N} + \beta
    \end{cases}
    \qquad &n \geq 2 \\
    \label{upb}
    \lambda_n \leq {{M}}_n
    &=
    \frac{\pi^2}{\mathcal{L}^2} \left( n - 2 + \mathcal{D} + \frac{\mathcal{N} + \beta}{2} + \beta \right)^2
    &n \in \mathbb{N},
\end{align}
If $\mathcal{D} \neq 0$ then the lower bound estimate holds for all $n \in \mathbb{N}$.
Otherwise if $\mathcal{D} = 0$ then $\lambda_1 = 0$.

In \cite{KuSe18} it was shown that (\ref{upb}) is attained by an infinite sequence of eigenvalues $\{\mu_{n_i} = {{M}}_{n_i} \}_{i \in \mathbb{N}}$ generated by a family of graphs with $\mathcal{D} = \mathcal{N} = 0$ and any $\beta \geq 2$, or $\beta = \mathcal{D} + \mathcal{N} = 1$, called respectively Windmill graphs, Neumann lasso graph and Dirichlet lasso graph.
It can be observed that these graphs also provide examples of sequences of eigenvalues which exhibit the equality in (\ref{lwb}), $\{\nu_{n_j} = {{m}}_{n_j} \}_{j \in \mathbb{N}}$.
Remarkably, the sequences $\mu_{n_i}$ and $\nu_{n_j}$ coincide.
This is possible because the indices ${n_i}$ and ${n_j}$ are respectively the smallest and the largest of a sequence of degenerate eigenvalues with multiplicity $m$:
\begin{equation}\label{eq:sharpdegenerate}
    i=j \iff \mu_{n_j} = \nu_{n_i} \iff n_j - n_i = m-1.
\end{equation}
We call {\bf lower sharp} and {\bf upper sharp} eigenvalues those which satisfy the equality in  (\ref{lwb}) and (\ref{upb}) respectively, and in general we call {\bf (degenerate) sharp eigenvalues} the eigenvalues of multiplicity $m \geq 1$ ($m \geq 2$) with smallest index upper sharp and largest index lower sharp like those in equation (\ref{eq:sharpdegenerate}).
The purpose of this work is to investigate which graphs exhibit sharp eigenvalues and discuss their properties.
The text is organized into three sections: Introduction and notation, Properties of sharp eigenvalues, and Main results.
We can summarize our findings as follows.

In Proposition~\ref{prp:sharpmult}, by direct comparison of (\ref{lwb}) and (\ref{upb}), we obtain an upper bound for the maximal eigenvalue multiplicity $m_\mathcal{U} = m_\mathcal{U}(\mathcal{G}) = \mathcal{D} + \mathcal{N} + 2 \beta - 1$.
Eigenvalues of multiplicity $m_\mathcal{U}$ are called {\bf maximally degenerate eigenvalues}.

In Theorem~\ref{thm:characterization} we show that sharp eigenvalues are characterized by being maximally degenerate.
This is used to show that sharp eigenvalues are preserved when multiple graphs are joined together at one of their Dirichlet pendant vertices (Lemma~\ref{lem:joinDirichlet}) or when a loop graph---with certain prescribed length---is attached to any Neumann pendant vertex (Lemma~\ref{lem:attachloop}).

In the proof of the main result, Theorem~\ref{thm:main}, we show how the aforementioned Lemmata can be used to construct a graph with arbitrary $\mathcal{N}, \mathcal{D}, \beta$ which produce sequences of degenerate sharp eigenvalues.
Graphs which can be constructed by recursive applications of Lemmata~\ref{lem:joinDirichlet} and~\ref{lem:attachloop} are trees where some of the pendant vertices have a loop graph attached, or, equivalently, trees decorated with some lasso graphs (also called tadpole or lollipop graphs), for this reason we call them {\bf lasso trees}.

By comparing $m_\mathcal{U}$ with the maximal eigenvalue multiplicity $m_\mathcal{M}$ proved in \cite{KaPi11} we observe that $m_\mathcal{U}(\mathcal{G}) \geq m_\mathcal{M}(\mathcal{G})$ with the equality occurring if and only if $\mathcal{G}$ is a lasso tree.
Finally we conclude in Theorem~\ref{thm:comparison} that sharp eigenvalues appear only in the spectra of lasso trees.

\section{Introduction and notation}

In this section we present the essentials about metric graphs used in the present text, for a more general introduction we refer to \cite{BeKu13} and \cite{Ku20}, see also \cite{KuSt02} and the recent \cite{Mu19}.

\paragraph{Metric graphs.}
Metric graphs are constructed as the quotient space of a set of distinct real intervals under an equivalence relation on the set of their endpoints.
Let ${E} = \sqcup_{i} e_i,\, e_i := [x_{2i-1},x_{2i}]$ be the disjoint union of closed real intervals and let $\sim$ be an equivalence relation over the endpoints of ${E}$.
The quotient space $\mathcal{G} = {E} / \sim$ is a metric graph, whose set of edges and vertices are ${E} = {E}(\mathcal{G})$ and, respectively, ${V} = {V}(\mathcal{G}) =  \sqcup_{i} \{ x_{2i-1},x_{2i} \} \,/ \sim$.
The metric and measure over $\mathcal{G}$ are inherited from the Euclidean metric and Lebesgue measure over the edges.
Moreover we denote by $\mathcal{L} = \sum_{i} |x_{2i} - x_{2i-1}|$ the total length of $\mathcal{G}$.
Since we are interested in metric graphs which are compact and with finite total length, we assume that $\mathcal{G}$ has a finite number of edges, $|{E}| < \infty$, each of them compact.
We allow the presence of loops and multiple edges.

\paragraph{Cycles and the first Betti number.}
A cycle is a finite sequence of distinct edges $\{ e^{(j)} \}_{j=1}^n$ associated to a sequence of distinct vertices $\{ v^{(j)} \}_{j=1}^{n}$ such that
\begin{itemize}
    \item $e^{(j)}$ is incident to $v^{(j)}$ and $v^{(j+1)}$ for $j=1,\dots,n-1$,
    \item $e^{(n)}$ is incident to $v^{(n)}$ and $v^{(1)}$.
\end{itemize}
A cycle of length $1$ is also called loop.
A graph with just one edge which is also a loop is called loop graph.

We denote by $\beta = \beta_1 = |{E}| - |{V}| + 1$, the first Betti number of $\mathcal{G}$;
$\beta$ coincides with the circuit rank of $\mathcal{G}$, namely the least number of edges that need to be removed in order to turn $\mathcal{G}$ into a tree.

\paragraph{Functions on metric graphs.}
Let $\textit{L}_2(\mathcal{G}) := \bigoplus_{i} \textit{L}_2 [x_{2i-1},x_{2i}]$. 
The space $\textit{L}_2(\mathcal{G})$ equipped with the inner product $\langle f , g \rangle := \sum_{i} \int_{e_i} f \overline{g}\, dx$ is a well defined Hilbert space.
If $f \in \textit{L}_2(\mathcal{G})$ is continuously differentiable over the edge $e_i = [x_{2i-1},x_{2i}]$, the oriented derivatives of $f$ at the endpoints of the interval $e_i$ are defined by $\partial f (x_{2i-1}) = f'(x_{2i-1})$ and $\partial f (x_{2i}) = -f'(x_{2i})$.
The oriented derivatives are well defined for functions in the Sobolev space $H^2(\mathcal{G}) := \bigoplus_{i} H^2 [x_{2i-1},x_{2i}]$.

\paragraph{Laplacian, vertex conditions, and eigenvalues.}
Given a subset of the pendant vertices $D$, which we call Dirichlet vertices, we define the Laplacian operator $L = L(\mathcal{G},D) := -\frac{d^2}{dx^2}$ with domain $\mathfrak{D}(L) = \mathfrak{D}(\mathcal{G},D)$ as the set of functions $f \in H^2(\mathcal{G})$ subject to the following conditions:
\begin{itemize}
    \item $f$ is continuous at the vertices, so $f \in \mathcal{C}(\mathcal{G})$ (continuity condition),
    \item the oriented derivatives of $f$ sum to zero at each non Dirichlet vertex: \\ $\sum_{x_j \in v} \partial f(x_j) = 0 ~~ \forall v \in {V} \setminus D$ (Kirchhoff condition),
    \item $f$ vanishes at the Dirichlet vertices, $f(v) = 0\,\forall v \in D$ (Dirichlet condition).
\end{itemize}

The continuity and Kirchhoff conditions together are called standard vertex conditions (in the literature sometimes called natural).

Standard vertex conditions at pendant vertices $v \notin D$ read as Neumann conditions: $f'(v) = 0$; hence we call these vertices Neumann and denote their set by $N$.
We denote by the corresponding calligraphic letter the cardinality of the two different type of pendant vertices $\mathcal{N} = |N|, \mathcal{D} = |D|$.

Standard vertex conditions at a vertex $v$ of degree two read as continuity of both the function and its derivative.
Let $\mathcal{G}$ have a pair of edges $e_i, e_j$ incident to a vertex $v$ of degree two and let $\mathcal{G}'$ be the graph where $e_i, e_j$ are replaced by a single edge with length equal to the sum of the lengths of $e_i, e_j$.
Then $\mathcal{G}',\textit{L}_2(\mathcal{G}')$, and $\mathfrak{D}(\mathcal{G}',D)$ are, respectively, isomorphic to $\mathcal{G},\textit{L}_2(\mathcal{G})$, and $\mathfrak{D}(\mathcal{G},D)$.
Hence, degree two vertices play no role in the study and can be freely removed whenever they occur in the construction of graphs.

Under the above hypothesis $L$ is a self-adjoint unbounded positive operator whose spectrum is exclusively discrete with unique accumulation point at $+\infty$, \cite{BeKu13}.
We denote the spectrum by $\sigma(L) = \{ \lambda_n \}_{n \in \mathbb{N}}$.
Whenever we say that an indexed eigenvalue $\lambda_n$ has multiplicity $m$ we assume $n$ to be the smallest index of the degenerate eigenvalue, i.e. $\lambda_n = \dots = \lambda_{n+m-1}$.
Moreover, unless differently stated, we shall assume $\mathcal{G}$ to be connected.
Under this hypothesis follows that the ground state $\lambda_1$ is simple (see \cite{Ku19}) and $\lambda_1 = 0$ if and only if $D = \emptyset$.
In \cite{BKKM17} the authors show that if $\mathcal{G}$ is not a loop graph then the eigenvalues of $L(\mathcal{G},D)$ satisfy the inequalities (\ref{lwb}), (\ref{upb}).

\paragraph{Quantum Graphs.}
The term quantum graph is used in general to refer to the triple $\Gamma = (\mathcal{G},-\frac{d^2}{dx^2} + q(x),\mathfrak{D}(\mathcal{G},D))$.
In the present setup $q \equiv 0$, so a metric graph $\mathcal{G}$ together with a set of Dirichlet vertices $D$ suffice to determine a quantum graph $\Gamma$.
Hence we use the generic term graph to refer to both metric and quantum graph whenever the set of Dirichlet vertices and the associated Laplacian are clear from the context.
In particular we may speak of the eigenvalues of a graph meaning the eigenvalues of the associated Laplacian.

In the study of the spectral estimates of quantum graphs several techniques have been developed, many of which put in relation modifications of the metric graph with the changes occurring in the spectrum.
Several of these are discussed in \cite{BKKM19} under the name of \emph{surgery principles}.
Lemmata \ref{lem:joinDirichlet} and \ref{lem:attachloop} make use of a particular case of two such principles.
In line with the previous paragraph, any modification brought to a graph $\mathcal{G}$ and its set of Dirichlet vertices $D$ should be reflected in the associated Laplacian.

\section{Properties of sharp eigenvalues}

\subsection{Sharp and maximally degenerate eigenvalues}
We start by making a general observation.

\begin{observation}\label{obs:strictineq}
    The sequences $\{{{m}}_n\}_{n \in \mathbb{N}}$ and $\{{{M}}_n\}_{n \in \mathbb{N}}$ given by (\ref{lwb}) and (\ref{upb}) are strictly increasing.
    As a consequence of this, if $\lambda_n = {{m}}_n$ then $\lambda_n < {{m}}_{n+1} \leq \lambda_{n+1}$; hence
    \begin{itemize}
        \item if $\lambda_n$ is lower sharp then $\lambda_n < \lambda_{n+1}$.
    \end{itemize}
    \noindent Similarly $\lambda_n = {{M}}_n$ implies that
    \begin{itemize}
        \item if $\lambda_n$ is upper sharp then $\lambda_{n-1} < \lambda_n$.
    \end{itemize}
\end{observation}

\begin{proposition}\label{prp:sharpmult}
    The multiplicity of any eigenvalue of the graph Laplacian is at most $\mathcal{D} + \mathcal{N} + 2\beta - 1$.
    Eigenvalues with maximal multiplicity are called maximally degenerate eigenvalues.
\end{proposition}
\begin{proof}
    Consider $\lambda_n$ be a degenerate eigenvalue of multiplicity $m \geq 2$, so $n \geq 2$ by the simplicity of the ground state.
    By direct application of the inequalities (\ref{lwb},\ref{upb}) to $\lambda_n$ and $\lambda_{n+m-1}$ we have
    \begin{equation}\label{eq:ineq_lambda_mu}
        {{m}}_{n+m-1} \leq {{M}}_n.
    \end{equation}
    Assuming $n+m-1 \geq \mathcal{N} + \beta$, then (\ref{eq:ineq_lambda_mu}) implies $m \leq \mathcal{D} + \mathcal{N} + 2\beta - 1$.
    
    If instead we assume $n+m-1 < \mathcal{N} + \beta$ then (\ref{eq:ineq_lambda_mu}) implies $m \leq n + 2\mathcal{D} + \mathcal{N} + 3\beta - 3$ which combined with the assumption leads to $m < \mathcal{D} + \mathcal{N} + 2\beta - 1$.
\end{proof}

\begin{lemma}\label{lem:sharpmult}
    Let $\lambda_n$ be a degenerate eigenvalue of multiplicity $m \geq 2$.
    Then $\lambda_n$ is a sharp degenerate eigenvalue 
    if and only if $\lambda_n$ is maximally degenerate $m = \mathcal{D} + \mathcal{N} + 2\beta - 1$.
\end{lemma}
\begin{proof}
    Because of the simplicity of the ground state, $n \geq 2$.
    Assume $n + m - 1 \geq \mathcal{N} + \beta$; hence by definition ${{m}}_{n + m - 1} = (\pi / \mathcal{L})^2 (n + m - 1 - (\mathcal{N} + \beta)/2)^2$ and
    \begin{align}\label{eq:lem_step1}
        {{M}}_n = {{m}}_{n + m - 1} &\iff n - 2 + \mathcal{D} + \frac{\mathcal{N} + \beta}{2} + \beta = n + m - 1 - \frac{\mathcal{N} + \beta}{2} \\ \nonumber
                                    &\iff m = \mathcal{D} + \mathcal{N} + 2\beta - 1.\\
    \intertext{Assume instead $n+m-1 < \mathcal{N} + \beta$; hence ${{m}}_{n + m - 1} = (\pi / \mathcal{L})^2 (n + m - 1)^2 / 4$; we show that this is not compatible with the hypothesis.
    In fact}
    \label{eq:lem_step2}
        {{M}}_n = {{m}}_{n + m - 1} &\iff n - 2 + \mathcal{D} + \frac{\mathcal{N} + \beta}{2} + \beta = \frac{n + m - 1}{2} \\
                                    &\iff m = n + 2\mathcal{D} + \mathcal{N} + 3\beta - 3, \nonumber
    \end{align}
    therefore $n+m-1 < \mathcal{N} + \beta$ reads $n + \mathcal{D} + \beta < 2$  which implies $n = 1$, a contradiction.
\end{proof}

The proof of Lemma~\ref{lem:sharpmult} suggests there might exist simple eigenvalues which are both upper sharp and lower sharp at the same time, which shall be called simple sharp eigenvalues.

Consider first $n \geq 2$: if ${{M}}_n = {{m}}_n$ then either
\begin{enumerate}[(i)]
    \item \label{case:1} $n < \mathcal{N} + \beta$, so by (\ref{eq:lem_step2}) $n=1$, which is excluded,
    \item \label{case:2} $n \geq \mathcal{N} + \beta$, so by (\ref{eq:lem_step1})
    $\mathcal{D} + \mathcal{N} + 2\beta = 2$.
\end{enumerate}
    
\noindent If we consider $n = 1$ then either

\begin{enumerate}[(i)]
  \setcounter{enumi}{2}
    \item \label{case:3}  $\mathcal{D} = 0$ thus $\lambda_1 = 0 = {{M}}_1$, so $\mathcal{N} + 3\beta = 2$,
    \item \label{case:4} $\mathcal{D} \neq 0$ and $1 < \mathcal{N} + \beta$, so by (\ref{eq:lem_step2}) $2\mathcal{D} + \mathcal{N} + 3\beta = 3$,
    \item \label{case:5} $\mathcal{D} \neq 0$ and $1 \geq \mathcal{N} + \beta$, so by (\ref{eq:lem_step1}) $\mathcal{D} + \mathcal{N} + 2\beta = 2$.
\end{enumerate}
Case (\ref{case:2}) is satisfied by any of the following:
\begin{itemize}
    \item $\beta = 1$ and $\mathcal{N} = \mathcal{D} = 0$, i.e. the loop graph which should be disregarded as it is an exceptional case for which neither (\ref{lwb}) nor (\ref{upb}) holds.
    \item $\beta = 0$ and $\mathcal{N} + \mathcal{D} = 2$, i.e. the interval with any admissible vertex conditions at its endpoints.
\end{itemize}
Case (\ref{case:3}) implies $\mathcal{N}=2, \beta = 0$, namely the Neumann-Neumann interval.
Case (\ref{case:4}) does not have solutions.
Case (\ref{case:5}) implies $\beta = 0$ and either $\mathcal{N} = 0, \mathcal{D} = 2$, or $\mathcal{N} = \mathcal{D} = 1$ hence the remaining two possible vertex conditions for the single interval.

Therefore the single interval with any of the admissible vertex conditions provides the only three examples of graphs with simple sharp eigenvalues, and in particular the whole spectrum is composed only by simple sharp eigenvalues:

\begin{itemize}
	\item $\mathcal{N}=2,\mathcal{D}=0$, the spectrum is $\lambda_n^{NN} = \frac{\pi^2}{\mathcal{L}^2}(n-1)^2$,
	\item $\mathcal{N}=1,\mathcal{D}=1$, the spectrum is $\lambda_n^{ND} = \frac{\pi^2}{\mathcal{L}^2}(n-\frac{1}{2})^2$,
	\item $\mathcal{N}=0,\mathcal{D}=2$, the spectrum is $\lambda_n^{DD} = \frac{\pi^2}{\mathcal{L}^2}n^2$.
\end{itemize}

\begin{proposition}\label{prp:sharpint}
    An eigenvalue is simple sharp if and only if the underlying graph is a single interval with any of the three possible combinations of vertex conditions listed above.
\end{proposition}

From now on we refer to as sharp eigenvalues the eigenvalues which are either simple or degenerate eigenvalues which are sharp.
The above discussion can be summarized by the following statement:
\begin{theorem}\label{thm:characterization}
    Let $\mathcal{G}$ be a graph which is not a cycle and let $\lambda$ be an eigenvalue of some Laplacian over $\mathcal{G}$.
    Then $\lambda$ is sharp if and only if it is maximally degenerate.
\end{theorem}

\subsection{Sharp eigenvalues and fully supported eigenspace}

In this section we show a necessary property of the eigenfunctions associated to degenerate sharp eigenvalues.
For its proof we need the following proposition about the regularity of the eigenvalues seen as functions dependent on the length of an edge of the graph.

\begin{proposition}\label{prp:continuouslambda}
    For any fixed index $n \in \mathbb{N}$ and edge $e$ of length $\ell$, the function $\ell \mapsto \lambda_n(\ell)$ is continuous on $(0,+\infty)$.
\end{proposition}

\begin{proof}
    Consider the Courant-Fischer eigenvalues characterizations via the Rayleigh quotient
    \begin{equation}\label{eq:rayleigh}
        \lambda_n
        = \min_{\substack{X \subset \mathfrak{D}_q \\ \dim(X) = n}} \max_{u \in X} \frac{\norm{u'}_2^2}{\norm{u}_2^2}
        = \norm{\psi_n'}_2^2,
    \end{equation}
    where $\mathfrak{D}_q = \mathfrak{D}_q(\mathcal{G},D) := \{f \in \textit{H}^{\,1}(\mathcal{G}) \cap \mathcal{C}(\mathcal{G}) : f(v) = 0 \,\,\forall v \in D \}$ and $\psi_n$ is any normalized eigenfunction associated to $\lambda_n$.
    In order to prove the statement we show that $\rho \mapsto \lambda_n(\rho\ell)$ is continuous in $\rho = 1$.
    Let $\mathcal{G}_\rho$ be the modification of $\mathcal{G}$ where the edge $e$ is stretched by a factor $\rho$, i.e. $e$ is replaced by $\rho e$ and consequently $\ell$ replaced by $\rho \ell$.
    Let $X \subset \mathfrak{D}_q$ be any subset realizing the minimum in (\ref{eq:rayleigh}).
    Let $X_\rho \subset \mathfrak{D}_q(\mathcal{G}_\rho)$ be the space obtained from $X$ by stretching each function over the edge $e$, i.e. $f_\rho(x) = f(x)$ if $x \in \mathcal{G}_\rho \setminus \rho e$ and $f_\rho(x) = f(x/\rho)$ if $x \in \rho e$.
    From the Rayleigh quotient it follows that
    \begin{equation}\label{eq:rayleighbound2}
        \lambda_n(\rho \ell) \leq \max_{u_\rho \in X_\rho} \frac{\norm{u_\rho'}_2^2}{\norm{u_\rho}_2^2}.
    \end{equation}
    We compute 
    \begin{equation}
    \begin{aligned}
        \norm{u_\rho'}_{L_2(\mathcal{G}_\rho)}^2
        &= \int_{\mathcal{G}_\rho \setminus \rho e} (u_\rho')^2\,dx + \int_{\rho e} (u_\rho')^2\,dx \\
        &=\int_{\mathcal{G} \setminus e} (u')^2\,dx + \int_{e} \left( \frac{1}{\rho} u' \right)^2 \rho \,dx \\
        &=\int_{\mathcal{G}} (u')^2\,dx + \left( \frac{1}{\rho} - 1 \right) \int_{e} \left( u' \right)^2 \,dx
    \end{aligned}
    \end{equation}
    and similarly we also obtain
    \begin{eqnarray}
        \norm{u_\rho}_{L_2(\mathcal{G}_\rho)}^2 =\int_{\mathcal{G}} u^2\,dx + \left( \rho - 1 \right) \int_{e} u^2 \,dx.
    \end{eqnarray}
    Therefore, if $\rho \leq 1$ we have the following upper estimate
    \begin{equation}
    \begin{aligned}
        \max_{u_\rho \in X_\rho} \frac{\norm{{u_\rho}'}_2^2}{\norm{u_\rho}_2^2} 
        &= \max_{u \in X} \frac{\norm{u'}_{L_2(\mathcal{G})}^2 + (\frac{1}{\rho} - 1)\norm{u'}_{L_2(e)}^2 }{\norm{u}_{L_2(\mathcal{G})}^2 + (\rho-1) \norm{u}_{L_2(e)}^2}  \\
        &\leq \frac{1}{\rho^2} \max_{u \in X} \frac{\norm{u'}_{L_2(\mathcal{G})}^2}{\norm{u}_{L_2(\mathcal{G})}^2}  \\
        &\leq \frac{1}{\rho^2} \lambda_n(\ell).
    \end{aligned}
    \end{equation}
    Moreover, from the monotonicity of the eigenvalues (see for example Corollary 3.12 in \cite{BKKM19}) we know that $\rho \leq 1 \Rightarrow \lambda_n(\ell) \leq \lambda_n(\rho \ell)$.
    
    Thus for $\rho \leq 1$ we have
    \begin{equation}
        \lambda_n(\ell) \leq \lambda_n(\rho \ell) \leq \frac{1}{\rho^2} \lambda_n(\ell).
    \end{equation}
    By changing $\ell$ with $\ell / \rho$ then we can deduce the more general inequality for any $\rho > 0$:
    \begin{equation}
        \min \left\{ 1, \rho^{-2} \right\} \lambda_n(\ell) \leq \lambda_n(\rho \ell) \leq \max \left\{ 1, \rho^{-2} \right\} \lambda_n(\ell),
    \end{equation}
    which shows the continuity in $\rho = 1$ of $\rho \mapsto \lambda_n(\rho\ell)$ and hence the claimed continuity of $\ell \mapsto \lambda_n(\ell)$ for $\ell \in (0,+\infty)$.
\end{proof}

The next lemma shows that both upper and lower sharp eigenvalues can be associated to eigenfunctions that do not identically vanish on any edge of the graph.

\begin{lemma}\label{lem:prop}
	If $\lambda_n$ is either lower or upper sharp then for each edge $e$ there exists an eigenfunction associated to $\lambda_n$ which is not identically zero on $e$.
\end{lemma}
\begin{proof}
    Assume $\lambda_n = {{M}}_n$ and let $\ell(e) = \ell_0$.
    Making $e$ longer increases the total length of the graph and consequently decreases ${{M}}_n$,  we write ${{M}}_n(\ell)$ to highlight the estimate dependence on the length of the edge $e$.
    In order not to violate (\ref{upb}), $\lambda_n$ must also decrease, at least as much as its upper bound.
    Assume $\lambda_n$ has multiplicity $m$; hence by Observation~\ref{obs:strictineq} $\lambda_{n-1} < \lambda_n = \dots = \lambda_{n+m-1} < \lambda_{n+m}$.
    By Proposition~\ref{prp:continuouslambda} all eigenvalues are continuous functions in the length $\ell(e)$, so there exists $\varepsilon > 0$ small such that $\forall \ell \in [\ell_0,\ell_0 + \varepsilon]$
    \begin{equation}\label{eq:cond_and}
        \lambda_{n-1}(\ell) < \lambda_n(\ell) \quad \textit{and} \quad \lambda_{n+m-1}(\ell) < \lambda_{n+m}(\ell).
    \end{equation}
    Let $\{ \psi_j\}_{j=n}^{n+m-1}$ be a basis of the $m$-dimensional eigenspace associated to $\lambda_n$.
    If each $\psi_j$ is identically zero on $e$, then $\psi_j$ is still an eigenfunction after perturbing the length of $e$ over the interval $[\ell_0,\ell_0 + \varepsilon]$ and by the Rayleigh quotient it is associated to an eigenvalue equal to $\lambda_n(\ell_0)$ with the same multiplicity $m$.
    Because of (\ref{eq:cond_and}) the indices of the eigenvalues are preserved; hence
    \begin{equation}
        \lambda_j(\ell) \equiv \lambda_j(\ell_0) \quad \forall \ell \in [\ell_0,\ell_0 + \varepsilon].
    \end{equation}
    This leads to the following contradiction 
    \begin{equation}
    \begin{aligned}
        {{M}}_n(\ell_0 + \varepsilon) < {{M}}_n(\ell_0) &= \lambda_n(\ell_0) \\ &= \lambda_n(\ell_0 + \varepsilon) \leq {{M}}_n(\ell_0 + \varepsilon).
    \end{aligned}
    \end{equation}
    Thus there exists an eigenfunction not identically zero on $e$.
\end{proof}
We then have the next corollary.
\begin{corollary}\label{cor:nonzeroeigf}
    If $\lambda$ is a sharp eigenvalue then there exists an eigenfunction associated to $\lambda$ which does not identically vanish on any edge of the graph.
\end{corollary}

\section{Main results}

The main theorem is proven in a constructive manner and relies upon the next two lemmata, each of them providing an operation which preserves sharp eigenvalues.
The first of them, Lemma~\ref{lem:joinDirichlet}, tells us that joining together graphs which share a sharp eigenvalue preserves not only the eigenvalue but also its sharpness.

Let $\{\mathcal{G}_i\}_{i=1}^p$ be a finite set of graphs, each of them with $\mathcal{N}_i$ Neumann, $\mathcal{D}_i \neq 0$ Dirichlet pendant vertices and $\beta_i$ first Betti numbers respectively.
For each $\mathcal{G}_i$ fix a Dirichlet vertex $v_i \in D_i$ (see for example the set of graphs on the left of figure \ref{fig:joinDirichlet}).
Assume that the spectrum of each $\mathcal{G}_i$ contains the same eigenvalue $\lambda$, not necessarily with the same index $\lambda_{n_i}(\mathcal{G}_i) = \lambda \,\forall i$.
Consider the graph $\mathcal{G}$ obtained by the disjoint union of all graphs $\bigsqcup_{i=1}^p \mathcal{G}_i$ with the 
vertices $v_i$ replaced by a single vertex endowed with standard vertex conditions as in Figure~\ref{fig:joinDirichlet}.
Then $\lambda$ is still an eigenvalue of $\mathcal{G}$.
We have then the following statement.

\begin{lemma}\label{lem:joinDirichlet}
    If $\lambda_{n_i} = \lambda$ is a sharp eigenvalue of each $\mathcal{G}_i$, then $\lambda_n = \lambda, n = 2 - p + \sum_{i=1}^p n_i$ is also a sharp eigenvalue of $\mathcal{G}$.
\end{lemma}

\begin{figure}[ht]
    \centering
    \includegraphics[width=.8\textwidth]{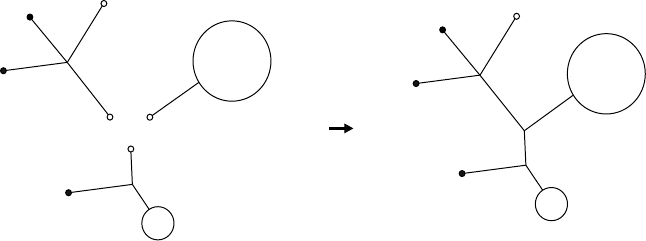}
    \caption{Example of application of Lemma~\ref{lem:joinDirichlet} to three graphs (left) joined at one chosen Dirichlet vertex for each of them in order to obtain the graph on the right.
    Here and in the following figures the symbol $\circ$ stands for a Dirichlet pendant vertex and $\bullet$ for a Neumann pendant vertex.}
    \label{fig:joinDirichlet}
\end{figure}

The above lemma allows us to build trees with sharp eigenvalues: one starts by joining intervals into star graphs with at least one Dirichlet pendant and then the star graphs into a tree.
This Lemma allows the construction of three graphs with any prescribed number of Dirichlet and Neumann pendant vertices which exhibit sharp eigenvalues.
It remains to show that it is as well possible to prescribe the first Betti number and still be able to construct a graph with sharp eigenvalues.
This is achieved by Lemma~\ref{lem:attachloop} which shows that sharp eigenvalues are preserved after attaching a cycle to a Neumann pendant.
We have already mentioned that the loop graph is the only graph which does not satisfy the inequalities (\ref{lwb}~\ref{upb}), in particular we can notice the following:

\begin{proposition}
	The spectrum of the loop graph $L_\mathcal{L}$ of length $\mathcal{L}$ is given by $\sigma(L_\mathcal{L}) = \{\lambda_1 = 0 \} \cup \left\{ \lambda_{2j} = \lambda_{2j+1} = \frac{\pi^2}{\mathcal{L}^2}(2j)^2  : j \in \mathbb{N}_{>0} \right\}$.
	The even eigenvalues of the loop graph exceeds the upper estimate (\ref{upb}) by a term $+\frac{1}{2}$ as follows:
	\begin{equation}
		\lambda_{2j}(L_\mathcal{L}) = \frac{\pi^2}{\mathcal{L}^2} \left( 2j - 2 + \frac{3}{2} + \frac{1}{2} \right)^2.
	\end{equation}
	The odd eigenvalues, excluded the first, differs from the lower estimate (\ref{lwb}) by a term $-\frac{1}{2}$ as follows:
	\begin{equation}
		\lambda_{2j+1}(L_\mathcal{L}) = \frac{\pi^2}{\mathcal{L}^2} \left( 2j + 1 - \frac{1}{2} - \frac{1}{2} \right)^2.
	\end{equation}
\end{proposition}

Therefore, given $\lambda > 0$ and $j \in \mathbb{N}$, the loop graph with length $\ell := 2 j \pi / \sqrt{\lambda}$ has the eigenvalue $\lambda_{2j} = \lambda$ with multiplicity $2$.
Now consider $\mathcal{G}$ any graph with at least one Neumann pendant vertex $v$ with a certain eigenvalue $\lambda_n(\mathcal{G}) = \lambda$.
Let $\mathcal{G}_v$ be the graph obtained by attaching the loop graph of length $\ell$ to the Neumann vertex $v$ with standard vertex conditions imposed there as in Figure~\ref{fig:attachloop}.
We have the following statement:

\begin{lemma}\label{lem:attachloop}
If $\lambda_n = \lambda$ is a sharp eigenvalue of $\mathcal{G}$ then $\lambda_{n_v} = \lambda, n_v = n+2j-1$ is a sharp degenerate eigenvalue of $\mathcal{G}_v$.
In particular, the multiplicity of $\lambda$ going from $\mathcal{G}$ to $\mathcal{G}_v$ increases by one.
\end{lemma}

\begin{figure}[ht]
    \centering
    \includegraphics[width=.8\textwidth]{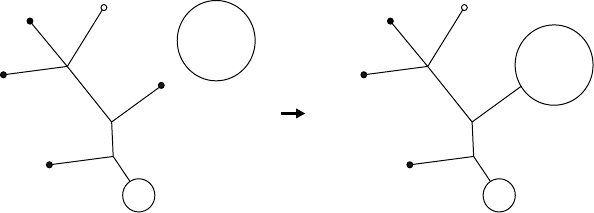}
    \caption{Example of the construction considered in Lemma~\ref{lem:attachloop}.
    On the left the graph $\mathcal{G}$ and the loop graph, on the right the graph $\mathcal{G}_v$.}
    \label{fig:attachloop}
\end{figure}

Lemmata~\ref{lem:joinDirichlet} and~\ref{lem:attachloop} applied to a set of intervals and loop graphs provide the tools to derive the main result, Figure~\ref{fig:thm_main} shows an example of graph constructed following the proof of Theorem~\ref{thm:main}.

\begin{theorem}\label{thm:main}
	Given $\mathcal{N},\mathcal{D}, \beta \in \mathbb{N} \cup \{0\}$ such that $\mathcal{N} + \mathcal{D} + \beta \geq 2$, there exists a graph with $\mathcal{N}$ Neumann, $\mathcal{D}$ Dirichlet pendant vertices respectively and first Betti number $\beta$ which exhibits an infinite sequence of sharp eigenvalues.
\end{theorem}

\begin{figure}[ht]
    \centering
    \includegraphics[width=.4\textwidth]{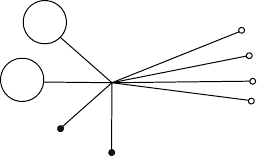}
    \caption{Example of graph constructed via Theorem~\ref{thm:main} with $\mathcal{D} = 4, \mathcal{N} = 2, \beta = 2$ with the lengths of the edges to scale.}
    \label{fig:thm_main}
\end{figure}

In \cite{KaPi11} the authors show that the maximal multiplicity of eigenvalues of the Schr\"odinger operator $-\frac{d^2}{dx^2} + q(x)$ with potential $q \in \textit{L}_1$ defined on a compact graph $\mathcal{G}$ is $m_\mathcal{M} = \mathcal{\beta} + \mathcal{P}^T - 1$, where $\mathcal{P}^T$ is the number of pendant vertices of the tree graph $T_\mathcal{G}$ obtained from $\mathcal{G}$ after contracting each cycle to a vertex.

We observe that the contraction of any cycle may generate at most one new pendant vertex, thus $\mathcal{P}^T - (\mathcal{D} + \mathcal{N}) \leq \beta$.
This means that $m_\mathcal{M} \leq m_\mathcal{U}$ with the equality occurring if and only if  $T_\mathcal{G}$ has exactly $\beta$ pendant vertices more than $\mathcal{G}$, or equivalently $\mathcal{G}$ is a lasso tree.
\begin{definition}
    A lasso tree is a compact metric graph where each cycle is a loop incident to a vertex of degree three.
\end{definition}
The previous observation together with Theorem \ref{thm:characterization} lead us to the following conclusion.

\begin{figure}[ht]
    \centering
    \includegraphics[width=.5\textwidth]{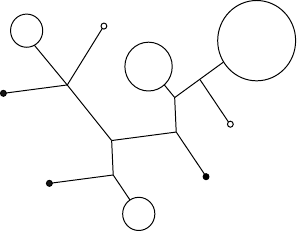}
    \caption{Example of a lasso tree.}
    \label{fig:lasso_tree}
\end{figure}

\begin{theorem}\label{thm:comparison}
    If $\mathcal{G}$ is a metric graph with sharp eigenvalues, then $\mathcal{G}$ is a lasso tree. 
\end{theorem}

\begin{remark}
    We point out that Lemmata \ref{lem:joinDirichlet} and \ref{lem:attachloop} can be applied recursively to construct lasso trees, with any possible topological structure, having sharp eigenvalues.
\end{remark}

\subsection{Proofs}

\begin{proof}[Proof of Lemma~\ref{lem:joinDirichlet}]
    By Theorem~\ref{thm:characterization} each $\lambda_{n_i}$ is maximally degenerate; hence with multiplicity $m_i = \mathcal{D}_i + \mathcal{N}_i + 2\beta_i - 1$.
	The spectrum of the disjoint union of the graphs $\{\mathcal{G}_i\}$ is the disjoint union of their eigenvalues, therefore $\lambda$ is an eigenvalue of $\bigsqcup_{i=1}^p \mathcal{G}_i$ with multiplicity $\sum m_i$.
	Since $\lambda_{n_i - 1} < \lambda_{n_i}$ then there are $(\sum n_i) - p$ strictly smaller eigenvalues than $\lambda$, possibly zero.
	Hence the smallest index of $\lambda$ on $\bigsqcup_{i=1}^p \mathcal{G}_i$ is $1 + \sum (n_i - 1)$.
	Replacing the vertices $\{v_i\}$ by a single vertex $v$ endowed with standard vertex conditions is an operation which increases the dimension of the domain of the quadratic form associated to the Laplacian by one, thus it corresponds to a rank one perturbation of the operator which pushes all the eigenvalues down, but no further than one index, i.e. it interlaces the eigenvalues $\lambda_{j-1}( \bigsqcup_{i=1}^p \mathcal{G}_i) \leq \lambda_j( \mathcal{G}) \leq \lambda_{j}( \bigsqcup_{i=1}^p \mathcal{G}_i)\, \forall j \in \mathbb{N}$.
	This operation can be seen as the inverse of a particular case of Theorem 3.4 (2) in \cite{BKKM19}, see also Theorem 3.1.8 in \cite{BeKu13}.
	Since $\lambda$ has multiplicity $\sum m_i$ on $\bigsqcup_{i=1}^p \mathcal{G}_i$ after the change of vertex condition, $\lambda$ is still an eigenvalue on $\mathcal{G}$, with multiplicity at least $m = (\sum m_i) - 1$ and correspondingly with lowest index at most $n = 2 + \sum (n_i - 1) $.
	We shall now show that $n$ and $m$ are indeed exact.
	Notice that when going from $\bigsqcup_{i=1}^p \mathcal{G}_i$ to $\mathcal{G}$ we have that
	\begin{itemize}
	    \item the number of Dirichlet pendant vertices is reduced by $p$,\\ $\mathcal{D} = (\sum_{i=1}^p \mathcal{D}_i) - p$;
	    \item the number of Neumann pendant vertices is preserved,\\ $\mathcal{N} = \sum_{i=1}^p \mathcal{N}_i$;
	    \item the first Betti number is preserved,\\ $\beta = \sum_{i=1}^p \beta_i$;
	\end{itemize}
	Therefore we compute that
	\begin{equation}
	    \begin{aligned}
	    m = \left( \sum m_i \right) - 1 &= \sum_{i=1}^n \left( \mathcal{D}_i + \mathcal{N}_i + 2\beta_i - 1 \right) - 1\\
	    &= \mathcal{D} + \mathcal{N} + 2\beta - 1
	    \end{aligned}
	\end{equation}
	which coincides with the maximal admissible multiplicity.
	Hence $\lambda$ must have precisely multiplicity $m$ and consequently  lowest index $n$.
	By Theorem~\ref{thm:characterization} $\lambda_n$ must be a sharp eigenvalue.
\end{proof}

\begin{proof}[Proof of Lemma~\ref{lem:attachloop}.]
	Consider the spectrum of the disjoint union of $\mathcal{G}$ and $L_\ell$, which is the disjoint union of their spectra.
	Then $\lambda$ has now smallest index $n_v = (n-1) + (2j - 1) + 1$ and if $\lambda$ has multiplicity $m$ on $\mathcal{G}$ then its multiplicity on $\mathcal{G} \sqcup L_\ell$ is $m + 2$.
	The action of attaching the loop graph to $\mathcal{G}$ at the vertex $v \in \mathcal{G}$ is a rank one perturbation of the graph Laplacian which decreases the domain of its associated quadratic form and consequently pushes the eigenvalues up, but no further than the eigenvalue of next index; hence
	\begin{equation}
	    \lambda = \lambda_{n + 2j - 1}(\mathcal{G} \sqcup L_\ell) \leq \lambda_{n + 2j - 1}(\mathcal{G}_v) \leq \lambda_{n+2j}(\mathcal{G} \sqcup L_\ell) = \lambda.
	\end{equation}
	We now show that after this operation the multiplicity of $\lambda$ is reduced by one, i.e. it is $m + 1$, and consequently the smallest index of $\lambda$ on $\mathcal{G}_v$ is still $n_v$.
	Let us parameterize the loop graph by the interval $[-\ell / 2,\ell / 2]$ with the zero placed in $v$.
	Any eigenfunction $\varphi$ on $\mathcal{G}$ can be extended to $\mathcal{G}_v$ by
	\begin{equation}
		\widetilde{\varphi}(x) :=
		\begin{cases}
		\varphi(x)                      &\text{if }x \in \mathcal{G},            \\
		\varphi(v)\cos(\sqrt{\lambda}x) &\text{if }x \in [-\ell / 2,+\ell / 2].
		\end{cases}
	\end{equation}
	So all the eigenfunctions on $\mathcal{G}$ associated to $\lambda_n$ are embedded in $\mathcal{G}_v$.
	In addition the following eigenfunction $\widetilde{\varphi}$ from $L_\ell$ can be embedded in $\mathcal{G}_v$
	\begin{equation}
		\widetilde{\varphi}(x) :=
		\begin{cases}
		    0                       &\text{if }x \in \mathcal{G},            \\
		    \sin(\sqrt{\lambda}x)	&\text{if }x \in [-\ell / 2,+\ell / 2].
		\end{cases}
	\end{equation}
	Hence going from $\mathcal{G} \sqcup L_\ell$ to $\mathcal{G}_v$ the multiplicity of $\lambda$ is reduced by one.
	Now notice that by Theorem~\ref{thm:characterization} the multiplicity of $\lambda$ on $\mathcal{G}$ is $m = \mathcal{D} + \mathcal{N} + 2 \beta - 1$, and hence
	\begin{equation}
        m + 1 = \mathcal{D} + (\mathcal{N}-1) + 2 (\beta+1) - 1.
	\end{equation}
	which is the maximal admissible eigenvalue multiplicity on $\mathcal{G}_v$ since this graph has one more cycle and one less Neumann pendant than $\mathcal{G}$.
	Again by Theorem~\ref{thm:characterization}, $\lambda$ must be a sharp degenerate eigenvalue of $\mathcal{G}_v$ with smallest index necessarily $n_v$.
\end{proof}

\begin{proof}[Proof of Theorem~\ref{thm:main}]
	Let $\mathcal{N}, \mathcal{D}$ and $\beta$ be given.
	In order to construct a graph with these corresponding numbers of Neumann pendants, Dirichlet pendants and first Betti number respectively, it is enough to consider
	\begin{itemize}
	    \item $\mathcal{N} + \beta$ copies of Neumann-Dirichlet intervals $I^{ND}$ of length $\ell_N = \pi / 2$,
	    \item $\mathcal{D}$ copies of Dirichlet-Dirichlet intervals $I^{DD}$ of length $\ell_D = \pi$,
	    \item $\beta$ copies of loop graph $L$ of length $\ell_L = 2\pi$.
	\end{itemize}
	Notice that the above three graphs share the following sequence of eigenvalues:
	\begin{equation}
	    \lambda^{ND}_{j} =
	    \lambda^{DD}_{2j-1} =
	    \lambda^{L}_{2(2j-1)} =
	    (2j - 1)^2.
	\end{equation}
    Apply Lemma~\ref{lem:joinDirichlet} to all the above intervals, both Neumann-Dirichlet and Dirichlet-Dirichlet to deduce that $\{ \lambda_{n_j} = ( 2j + 1 )^2 \}_{j \in \mathbb{N}}$ is a sequence of sharp eigenvalues, each of multiplicity $\mathcal{N} + \beta + \mathcal{D} - 1$, where
    \begin{equation}
        \begin{aligned}
        n_j &= 2 - ((\mathcal{N} + \beta) + \mathcal{D})  + (\mathcal{N} + \beta) \cdot j + \mathcal{D} \cdot (2j - 1) \\
            &= 2 - (\mathcal{N} + \beta) + (\mathcal{N} + \beta + 2 \mathcal{D}) j.
        \end{aligned}
    \end{equation}
	Notice that the length of the loop graph $L$ can be rewritten as
	\begin{equation}
	    \ell_L
	    = 2 \pi
	    = \frac{\pi \cdot 2(2j - 1)}{\sqrt{\lambda^{L}_{2(2j-1)}}}.
	\end{equation}
	Therefore we can recursively apply Lemma~\ref{lem:attachloop} $\beta$ number of times and obtain the new sequence of sharp eigenvalues $\{ \lambda_{\widetilde{n}_j} = (2j - 1)^2 \}$, each of multiplicity $\mathcal{N} + \mathcal{D} + 2 \beta - 1$, which is maximal, where
	\begin{equation}
	    \begin{aligned}
        \widetilde{n}_j
	    &= n_j + 2(2j - 1)\beta - \beta \\
	    &= 2 - (\mathcal{N} + 4\beta) + (\mathcal{N} + 4\beta + 2 \mathcal{D}) j.
	    \end{aligned}
	\end{equation}
\end{proof}

\begin{observation}
    The proof of Theorem~\ref{thm:main} with $\mathcal{N} = \mathcal{D} = 0$ recovers the family of Windmill graphs defined \cite{KuSe18}.
\end{observation}

\section*{Acknowledgement}

The author wishes to thank Pavel Kurasov for valuable suggestions and guiding, and Jacob Muller.

\bibliography{Arxiv_manuscript}{}
\bibliographystyle{unsrt}

\end{document}